\documentclass[12pt,english,reqno]{amsart}  

\usepackage[T1]{fontenc} 
\usepackage{geometry} 
\usepackage{amssymb,amsmath,amsthm} 

\usepackage{pdfpages}
\usepackage{graphicx}
\usepackage{tikz} 
\usetikzlibrary{arrows,automata} 

\usepackage{hyperref}
\hypersetup{
    colorlinks=true,
    linkcolor=blue,
    filecolor=magenta,      
    urlcolor=cyan,
}

\usepackage[utf8]{inputenc}

\newtheorem{theorem}{Theorem}[section]
\newtheorem{corollary}[theorem]{Corollary}
\newtheorem{lemma}[theorem]{Lemma}

\newtheorem{conjecture}[theorem]{Conjecture}

\newtheorem{remark}[theorem]{Remark}

\newcommand{\bfS}[1]{\mathbb{S}^{#1}} 
\newcommand{\bfR}[1]{\mathbb{R}^{#1}} 

\newcommand{\mbS}{\mathbb S}

\title[Stability Index and Yau's Conjecture for Carlotto-Schulz examples]{The stability index and Yau's conjecture for Carlotto-Schulz minimal hypertori, part II}
\author{Oscar Perdomo}
\address{Central Connecticut State University}
\email{perdomoosm@ccsu.edu}
\date{\today}

\begin{document}

\maketitle

\begin{abstract}
For any closed minimal hypersurface $M$ in the $N+1$-dimensional Euclidean sphere $\mbS^{N+1}$, $-N$ is an eigenvalue of the stability operator. In this paper we show that the multiplicity of this eigenvalue for the Carlotto and Schulz minimal embedding \(X_{CS}^{n}:\bfS{n-1}\times\bfS{n-1}\times\bfS{1}\to \bfS{2n}\) is at least $2n+1+n^2$. We conjecture that if $n>2$, then the stability index of $X_{CS}^{n}$ is $\frac{1}{3} \left(n^3+9 n^2+11 n+3\right)$ and for the hypertorus in $\mbS^4$ (case $n=2$) the stability index is $27$. We numerically verify the conjecture for the first 100 values of $n$.  We also numerically verify that Yau's conjecture on the first eigenvalue of the Laplacian holds when $2\le n\le 260$.

\end{abstract}

\section{Introduction}

Let $M$ be an $N$-dimensional compact manifold and let us assume that $X=(x_1,\dots,x_{N+2}):M\to \mbS^{N+1} $ is a minimal immersion. The stability operator on \(M\) is given by
\[
J = -\Delta - |A|^2 - N,
\]
where \(|A|^2 = \kappa_1^2 + \dots + \kappa_N^2\) and $\kappa_i$ are the principal curvatures of $M$. If \(\nu =(\nu_1,\dots,\nu_{N+2}): M \to \bfR{N+2}\) denotes its Gauss map, then 

\begin{eqnarray}\label{knowE}
J(\nu_i)=-N\nu_i\quad\hbox{and}\quad J(f_{ij})=0\, f_{ij}\quad {if}\quad f_{ij}=\nu_ix_j-\nu_jx_i .
\end{eqnarray}

The {\it stability index} of \(M\) is the number of negative eigenvalues (counted with multiplicity) of the stability operator. It is not difficult to show that when \(M\) is not totally geodesic, then the functions $\nu_i$ are linearly independent and therefore, these functions
contribute  $N+2$ to the stability index. For all the non-totally geodesic examples with known stability  spectrum, the multiplicity of $-N$ is $N+2$. In this paper we show that the multiplicity of $-(2n-1)$ for the minimal immersions
\(X_{CS}^{n}:\bfS{n-1}\times\bfS{n-1}\times\bfS{1}\to \bfS{2n}\) found by Carlotto and Schulz in \cite{CS} is at least $2n+1+n^2$. Regarding their stability index, we introduce  the following conjecture

\begin{conjecture}
If $n>2$, then the stability index of $X_{CS}^{n}$ is $\frac{1}{3} \left(n^3+9 n^2+11 n+3\right)$ and for the hypertorus in $\mbS^4$ (case $n=2$) the stability index is $27$. 
\end{conjecture}

We numerically show the conjecture for the first 100 values of $n$. 

In regard to the eigenvalues of the Laplacian, Yau has conjectured that the first nonzero eigenvalue of $-\Delta$ of any embedded minimal hypersurface in $\mbS^{N+1}$ is $N$. We also numerically check Yau's conjecture for the Carlotto-Schulz minimal embedding for $2\le n\le 260$. 

This paper can be viewed as a continuation of the work in \cite{P}, where the author shows that the stability index of $X_{CS}^{n}$ is at least \(n^2+4n+3\).

\section{Preliminaries}

Let us start this section by recalling the construction of the immersions $X_{CS}^n$ discovered in \cite{CS}. For any integer $n>1$, consider the immersion $X^n_{CS}:\bfS{n-1}\times \bfS{n-1}\times \bfS{1}\to \bfS{2n}\subset\bfR{2n+1}$ given by
\begin{eqnarray}\label{eq:Immersion}
X(y,z,t)=\bigl(\sin r(t)\,\cos\theta(t)\,y,\;\sin r(t)\,\sin\theta(t)\,z,\;\cos r(t)\bigr),
\end{eqnarray}
where $y,z\in\bfS{n-1}$. Under the assumption that the curve
\begin{eqnarray}\label{eq:ImmersionComponents}
\gamma(t)=\bigl(\gamma_1(t),\;\gamma_2(t),\;\gamma_3(t)\bigr)=\bigl(\sin r(t)\,\cos\theta(t),\;\sin r(t)\,\sin\theta(t),\;\cos r(t)\bigr)
\end{eqnarray}
is parametrized by arc length, there exists a function $\alpha(t)$ such that
\begin{equation}\label{eq:arclength}
\begin{aligned}
r'(t)&=\cos\alpha(t),\\
\theta'(t)&=\dfrac{\sin\alpha(t)}{\sin r(t)}.
\end{aligned}
\end{equation}
The curve $\gamma$ is called the \emph{profile curve}. A direct computation shows that the Gauss map of $X$ has the form
\[
\nu(t,y,z)=\bigl(\nu_1(t)\,y,\;\nu_2(t)\,z,\;\nu_3(t)\bigr),
\]
with
\begin{equation}\label{eq:GaussComponents}
\begin{aligned}
\nu_1(t)&=\cos r(t)\,\sin\alpha(t)\,\cos\theta(t)+\cos\alpha(t)\,\sin\theta(t),\\
\nu_2(t)&=\cos r(t)\,\sin\alpha(t)\,\sin\theta(t)-\cos\alpha(t)\,\cos\theta(t),\\
\nu_3(t)&=-\sin r(t)\,\sin\alpha(t).
\end{aligned}
\end{equation}
As shown in \cite{CS}, the immersion $X$ is minimal if and only if
\begin{equation}\label{eq:minimality}
\alpha'(t)=(2n-2)\,\csc r(t)\,\cos\alpha(t)\,\cot\bigl(2\theta(t)\bigr)
-(2n-1)\,\cot r(t)\,\sin\alpha(t).
\end{equation}
Using \eqref{eq:arclength} and \eqref{eq:minimality} yields the principal curvatures
\begin{align*}
\kappa_u(t)&=\csc r(t)\,\cos\alpha(t)\,\tan\theta(t)+\cot r(t)\,\sin\alpha(t),\\
\kappa_v(t)&=\cot r(t)\,\sin\alpha(t)-\csc r(t)\,\cos\alpha(t)\,\cot\theta(t),\\
\kappa_t(t)&=2\,\csc r(t)\,\cos\alpha(t)\,\cot\bigl(2\theta(t)\bigr)-2\,\cot r(t)\,\sin\alpha(t).
\end{align*}
Therefore, the square of the norm of the shape operator is
\begin{eqnarray}\label{nsA}
|A|^2=(n-1)\,\kappa_u^2+(n-1)\,\kappa_v^2+\kappa_t^2.
\end{eqnarray}
The existence of the embedding \(X\) follows from the following theorem in \cite{CS}:

\begin{theorem}[Existence of the  Carlotto–Schulz  embedding \cite{CS}]\label{ThmCS}
For any integer $n>1$, there exist \(s^* = T/4\) and \(r_0\in(0,\pi)\) such that the unique solution 
\(\bigl(r(t),\theta(t),\alpha(t)\bigr)\) of the system 
\eqref{eq:arclength}–\eqref{eq:minimality}, with initial conditions
\[
\theta(0)=\frac{\pi}{4},\quad
r(0)=r_0,\quad
\alpha(0)=-\frac{\pi}{2},
\]
satisfies at \(t=s^*\)
\[
\theta(s^*)>0,\quad
r(s^*)=\frac{\pi}{2},\quad
\alpha(s^*)=0.
\]
Moreover, on the interval \([0,s^*]\), the function \(\theta(t)\) is strictly decreasing, while \(r(t)\) and \(\alpha(t)\) are strictly increasing.
\end{theorem}

\begin{remark}\label{remark on solution}
Due to the symmetries of the ODE system, the graph of \(\theta(t)\) is symmetric about the vertical lines \(t=\tfrac{T}{4}\) and \(t=\tfrac{3T}{4}\). It also has odd symmetry with respect to the points \(\bigl(\tfrac{T}{2},\tfrac{\pi}{4}\bigr)\) and \(\bigl(0,\tfrac{\pi}{4}\bigr)\). Likewise, the graph of \(r(t)\) is symmetric about the lines \(t=0\) and \(t=\tfrac{T}{2}\) and has odd symmetry with respect to the points \(\bigl(\tfrac{T}{4},\tfrac{\pi}{2}\bigr)\) and \(\bigl(\tfrac{3T}{4},\tfrac{\pi}{2}\bigr)\). Consequently, \(\theta(t)\) and \(r(t)\) are \(T\)-periodic. The function \(\alpha(t)\) is not periodic, but it satisfies \(\alpha(t+T)=\alpha(t)+2\pi\). These observations show that  $\gamma(t)$ is a $T$-periodic embedded closed curve in $\bfS{2}$ and therefore $X_{CS}^n$ is not only immersed but embedded.
\end{remark}

Let us continue this section with some results on the spectrum of the Laplacian. Spheres are one of the few examples for which the spectrum of the Laplacian is completely known. We have:

\begin{theorem}\label{SpectrumSphere}  Let $\bfS{k}$ denote the $k$-dimensional unit sphere endowed with the standard metric as a subset of $\bfR{k+1}$. The eigenvalues of the Laplace operator $-\Delta$ are

$$\alpha_1=0,\quad \alpha_2=k, \quad \hbox{in general,} \quad \alpha_i=(i-1)(k+i-2)$$

with multiplicities

$$m_1=1,\quad m_2=k+1, \quad \hbox{and for $i>2$,} \quad m_i=\binom{k+i-1}{i-1}-\binom{k+i-3}{i-3}$$

\end{theorem}

The proof of the following lemma can be found in  \cite{P}.

\begin{lemma}\label{formula-laplacian}
Let $\zeta(t,y,z)$ be a smooth function defined along the immersion $X_{XC}^n: \bfS{n-1} \times \bfS{n-1} \times \bfS{1} \to \bfS{2n}$. Then
\[
\Delta\zeta
=\zeta_{tt}
+\tfrac{(n-1)}{2}\,\bigl(\ln(EG)\bigr)'\,\zeta_t
+\frac{1}{E}\,\Delta_{\bfS{n-1}}^y \zeta
+\frac{1}{G}\,\Delta_{\bfS{n-1}}^z \zeta,
\]
where 
\[
E = \sin^2(r(t)) \cos^2(\theta(t)), 
\quad 
G = \sin^2(r(t)) \sin^2(\theta(t)),
\]
and $\Delta_{\bfS{n-1}}^y \zeta$ (resp.\ $\Delta_{\bfS{n-1}}^z \zeta$) denotes the Laplacian on the sphere $\bfS{n-1}$ acting on the $y$-variable (resp.\ the $z$-variable), with $t$ and $z$ (resp.\ $t$ and $y$) held fixed.
\end{lemma}

\section{Unexpected multiplicity of $-(2n-1)$}

Using the notation in Lemma \ref{formula-laplacian}, let us denote by $y_1,\dots y_{n}$ the coordinate function of the natural immersion of $\bfS{n-1}$ into $\bfR{n}$ and $z_1,\dots z_{n}$ the coordinate function of the natural immersion of another copy of $\bfS{n-1}$ into $\bfR{n}$. It is well-known that 

$$ \Delta_{\bfS{n-1}}^y y_i=-(n-1) y_i \qquad\hbox{and}\qquad  \Delta_{\bfS{n-1}}^z z_i=-(n-1) z_i$$

\begin{lemma}\label{lemma1}  Let $r(t)$ and $\theta(t)$ be the functions that define the immersion $X_{CS}^n$ described in Theorem \ref{ThmCS}. For any $i, j\in\{1,\dots, n\}$, the functions   

$$\zeta_{ij}=f(t)y_iz_j \quad\hbox{with}\quad f = \left(\csc^2 r\sec{\theta}\csc\theta\right)^{n-1}$$ 

satisfy that

$$J(\zeta_{ij})=-(2n-1)\zeta_{ij}$$
\end{lemma}

\begin{proof}
The proof is a direct computation using \eqref{eq:arclength} and \eqref{eq:minimality}. We have that 

\begin{eqnarray*}
\frac{df}{dt}=(1-n)  \left(\cos r \cos \alpha  \sin (2 \theta)-\sin \alpha \sin ^2\theta+\sin \alpha  \cos ^2 \theta \right)\csc ^{2 n-1}(r) \csc^n \theta  \sec^n \theta
\end{eqnarray*} 

and 

$$\frac{d^2f}{dt^2}=(n-1)  \tan ^3\theta  \csc ^{2 n}(r) \csc^n \theta  \sec^n \theta \left(\frac{1}{8} \cos^2 \alpha \cot^2\theta \csc^2\theta\, w_1+\cot^3\theta\,  w_2+\cot^3\theta\, w_3 \right)$$

with
\begin{eqnarray*}
w_1&=&4 (n-1) \cos (2 r) \sin ^2(2 \theta)+(4-6 n) \cos (4 \theta)-2 n+4\\
w_2&=&\cos r \sin (2 \alpha)+2 \sin ^2\alpha  \cot \theta\\
w_3&=& \left(n-\left((2 n-2) \cos ^2r+2n-2\right)\cot ^2\theta+(n-2) \cot ^4\theta \right) \sin^2\alpha +\\
& & \left(4 (n-1) \cot ^2\theta-4 n+2\right)\cos r \sin (2 \alpha) \cot \theta 
\end{eqnarray*} 

We have that 

$$\frac{d^2f}{dt^2}+\tfrac{(n-1)}{2}\,\bigl(\ln(EG)\bigr)'\,\frac{df}{dt}+|A|^2 f
-\frac{n-1}{E} f-
\frac{n-1}{G}f=0
$$

Using the equation above we obtain that 

\begin{eqnarray*}
J(\zeta_{ij})&=& -\Delta(fy_iz_j)-|A|^2fy_iz_j-(2n-1) fy_iz_j\\
&=& -\frac{d^2f}{dt^2}y_iz_j-\tfrac{(n-1)}{2}\,\bigl(\ln(EG)\bigr)'\,\frac{df}{dt}y_iz_j-\frac{z_jf}{E}\,\Delta_{\bfS{n-1}}^y y_i \\
& &-\frac{y_if}{G}\,\Delta_{\bfS{n-1}}^z z_j-|A|^2fy_iz_j-(2n-1)fy_iz_j\\
&=& -\frac{d^2f}{dt^2}y_iz_j-\tfrac{(n-1)}{2}\,\bigl(\ln(EG)\bigr)'\,\frac{df}{dt}y_iz_j+\frac{n-1}{E}\, y_i z_jf\\
& &+\frac{n-1}{G} z_jy_if-|A|^2fy_iz_j-(2n-1)fy_iz_j\\
&= &-(2n-1)\zeta_{ij}
\end{eqnarray*}

\end{proof}

Since the function $f>0$ and the functions $\nu_1 y_i$, $\nu_2z_j$, $\nu_3$, and $fy_iz_j$ are linearly independent, we have

\begin{corollary} The multiplicity of $-(2n-1)$ as an eigenvalue of the stability operator $J$ is at least $n^2+2n+1$.
\end{corollary}

\section{Conjecture on the stability index on $X^n_{CS}$}

From the work in \cite{P}, we have that the set of  eigenvalues of the stability operator of the immersions $X_{CS}^n$ is the union of the eigenvalues of the operator  $S_{ij}$ where  $S_{ij}=\mathcal{L}_{ij}-|A|^2-(2n-1)$ with

\[
L_{i,j}[\eta]
= -\eta''(t)
 - \tfrac{n-1}{2}\bigl(\ln(EG)\bigr)'(t)\,\eta'(t)
 + \frac{\alpha_i}{E(t)}\,\eta(t)
 + \frac{\alpha_j}{G(t)}\,\eta(t),
\]

If we denote the eigenvalues of $S_{ij}$ by $\lambda_{ij}^1<\lambda_{ij}^2\le \lambda_{ij}^3\le\dots$ we have that
each one of these eigenvalues must be counted as $m_im_j$ eigenvalues when viewed as an eigenvalue of the stability operator of the immersion $X_{CS}^n$. In other words, each one of these eigenvalues $\lambda$ contributes with $m_im_j$ toward the stability index. Recall that the dimension of the manifold is $2n-1$ and the $\alpha_i$ and $m_i$ are, respectively,  the eigenvalues and multiplicities for the spectrum of the Laplacian of the $n-1$-dimensional sphere.
Some values that will be relevant are 

$$\alpha_1=0,\alpha_2=(n-1), \alpha_3=2n, \alpha_4=3(n+1), \alpha_5=4(n+2)$$

$$m_1=1, m_2=n, m_3=\frac{n^2+n-2}{2},m_4=\frac{n}{6}(n+4)(n-1), m_5=\frac{n}{24}(n-1)(n^2+7n+6)$$

Regarding the eigenvalues of the stability operator of the immersions $X_{CS}^n$ that contribute to the stability index, below are those  that have been mathematically shown to exist:

\begin{itemize}
\item
It is known that $\lambda_{11}^3=-(2n-1)$, therefore $\lambda_{11}^1$ and$\lambda_{11}^2$ are also negative and the operator $S_{11}$ contributes with at least $3$ eigenvalues to the index.
\item
It is known that $\lambda_{21}^2=-(2n-1)$,  and $\lambda_{21}^3=0$, therefore there are exactly two negative eigenvalues for $S_{21}$. Since $m_2=n$ and $m_1=1$ then the operator $S_{21}$ contributes exactly with 
$2n$  towards the index. Likewise $S_{12}$ also contributes with exactly $2n$ eigenvalues to the index.
\item
From Lemma \ref{lemma1}, it follows that $\lambda_{22}^1=-(2n-1)$. Since $m_2=n$, then, this eigenvalue contributes with $n^2$ toward the index. 
\item
The previous three items provide that the index is at least $n^2+4n+3$. One of the main results in \cite{P}.
\item
For the hypertorus (case $n=2$) we can numerically check that $\lambda_{11}^4>0$, $\lambda_{22}^2=0$. Moreover we can check that  $\lambda_{31}^1$, $\lambda_{13}^1$, $\lambda_{41}^1$, $\lambda_{14}^1$, $\lambda_{51}^1$, $\lambda_{15}^1$ are negative  and 
$\lambda_{31}^2$, $\lambda_{13}^2$, $\lambda_{41}^2$, $\lambda_{14}^2$, $\lambda_{51}^2$, $\lambda_{15}^2$,  $\lambda_{16}^1$ and $\lambda_{61}^1$ are positive. Therefore, the stability index of the hypertorus is
{\smaller
$$n^2+4n+3+2 * \frac{n^2+n-2}{2}+2*\frac{n}{6}(n+4)(n-1)+2*\frac{n}{24}(n-1)(n^2+7n+6)$$
}
which is 27 after replacing $n$ with 2.

\item
For the cases $n=3,\ldots,100$ we can numerically check that $\lambda_{11}^4>0$, $\lambda_{22}^2=0$. Moreover we can check that  $\lambda_{31}^1$, $\lambda_{13}^1$, $\lambda_{41}^1$, $\lambda_{14}^1$ are negative  and 
$\lambda_{31}^2$, $\lambda_{13}^2$, $\lambda_{41}^2$, $\lambda_{14}^2$,  $\lambda_{15}^1$ and $\lambda_{51}^1$ are positive. Therefore, numerically for $n=3,\dots 100$,  the stability index of the minimal immersion $X_{CS}^n$ 
is 
{\smaller
$$n^2+4n+3+2 * \frac{n^2+n-2}{2}+2*\frac{n}{6}(n+4)(n-1)=\frac{1}{3} \left(n^3+9 n^2+11 n+3\right)$$
}

\end{itemize}

From the considerations above, we propose the following conjecture:

\begin{conjecture}
If $n>2$,  the stability index of $X_{CS}^{n}$ is $\frac{1}{3} \left(n^3+9 n^2+11 n+3\right)$ and for the hypertorus in $\mbS^4$ (case $n=2$) the stability index is $27$. 
\end{conjecture}

\section{Numerical verification}

For each one of the operators
\[
S_{ij}(\eta)=-\eta^{\prime\prime}-b^\prime \eta^\prime-c\eta,
\]
where $b$ and $c$ are $T$-periodic functions, the eigenvalues of $S_{ij}$ agree with the values of $\lambda$ that satisfy
\[
\delta_{ij}(\lambda)=z_1(T)+z_2'(T)-2=0.
\]

where $z_1(t)$ and $z_2(t)$ satisfy  

\begin{eqnarray*}
z_1^{\prime\prime}+b^\prime z_1^\prime+cz_1+\lambda z_1&=&0\quad \hbox{with} \quad z_1(0)=1\quad\hbox{and} \quad z_1^\prime(0)=0\\
z_2^{\prime\prime}+b^\prime z_2^\prime+cz_2+\lambda z_2&=&0\quad \hbox{with} \quad z_2(0)=0\quad\hbox{and} \quad z_2^\prime(0)=1
\end{eqnarray*}

\subsection{The case $n=2$, the hypertorus in $\bfS{4}$} For this case we are taking $r(0)=1.2321501$ and $T=4\times 0.72803518$. We already know that $\lambda_{11}^3=-3=-(2n-1)$. We need to check that $\lambda_{11}^4>0$. 
Since the periodic eigenvalues of $S_{11}$ correspond to the zeros of the Floquet discriminant $\delta_{11}$, it suffices to show that $\delta_{11}$ has no zeros in the interval $(-3,0)$.

To this end, we numerically evaluate $\delta_{11}$ at the 500 equally spaced values

\[
{-4,-4+0.01,-4+2(0.01),\dots,-4+500(0.01)}.
\]

Recall that, for a given value of $\lambda$, the quantity $\delta_{11}(\lambda)$ is obtained by solving two second-order differential equations with prescribed initial conditions and evaluating the corresponding solutions at $t=T$.

Figure \ref{neq2S11} shows the graph of $\delta_{11}$. As expected, \(
\delta_{11}(-3)=0.
\) Moreover, the numerical computation shows that \(
\delta_{11}(\lambda)<0
\) for every sampled value of $\lambda$ in the interval $(-3,1)$. Therefore $\delta_{11}$ has no zeros in $(-3,0)$, and consequently \(
\lambda_{11}^{4}>0.
\) It follows that $S_{11}$ has exactly three negative eigenvalues.

\begin{figure}[ht]
  \centering
  \includegraphics[width=0.6\textwidth]{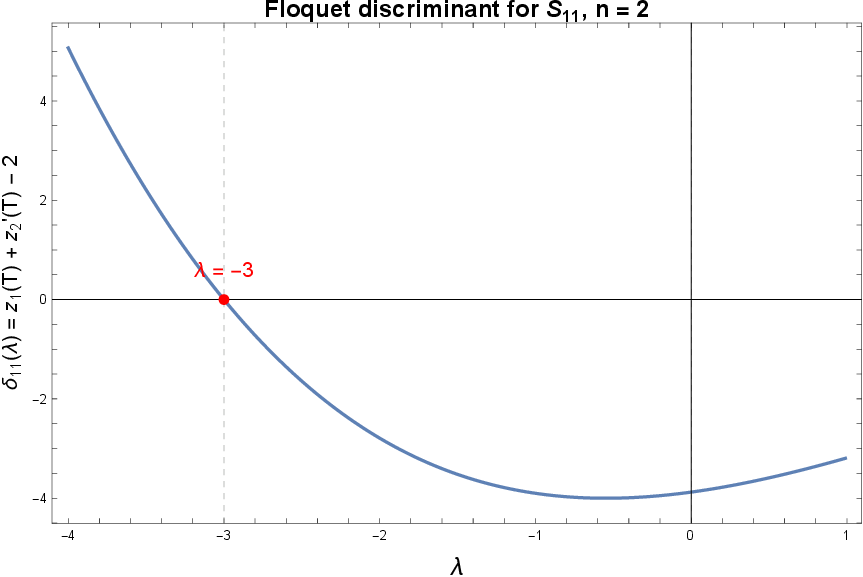}
  \caption{Graph of the Floquet discriminant $\delta_{11}$. The discriminant vanishes at $\lambda=-3$, and remains negative on $(-3,1)$. Hence $\lambda_{11}^{4}>0$.}
  \label{neq2S11}
\end{figure}

Figure \ref{neq2S22} shows that $S_{22}$ has only one negative eigenvalue. We already knew that $-3$ was the first eigenvalue of $S_{22}$. Even though we know that $S_{21}$ has exactly $2$ negative  eigenvalues, we  plotted the graph of the function $\delta_{21}$ (see Figure \ref{neq2S21}) to show that the first eigenvalue of $S_{21}$ is between $-19$ and $-18$. 

\begin{figure}[ht]
  \centering
  \includegraphics[width=0.6\textwidth]{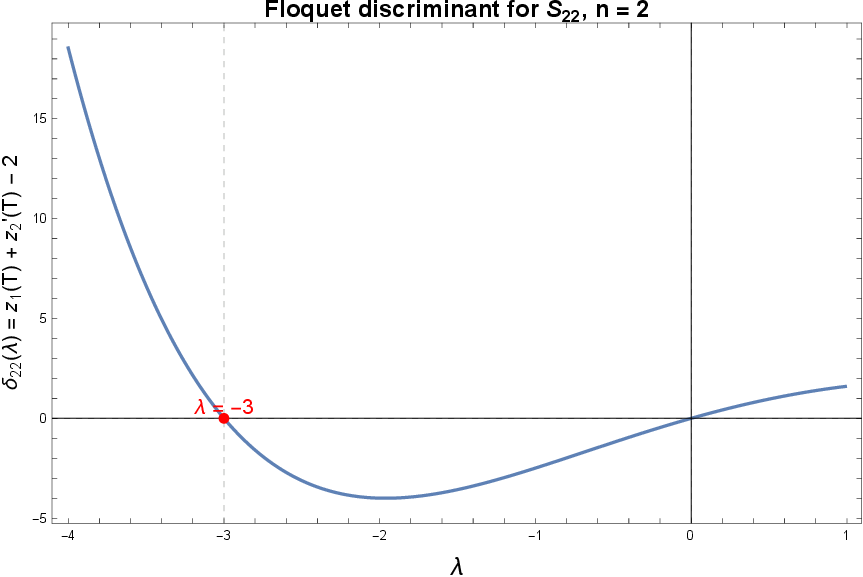}
  \caption{Graph of the Floquet discriminant $\delta_{22}$. The discriminant vanishes at $\lambda=-3$ and $\lambda=0$, and remains negative on $(-3,1)$. Hence $\lambda_{22}^{2}=0$.}
  \label{neq2S22}
\end{figure} 

\begin{figure}[ht]
  \centering
  \includegraphics[width=0.6\textwidth]{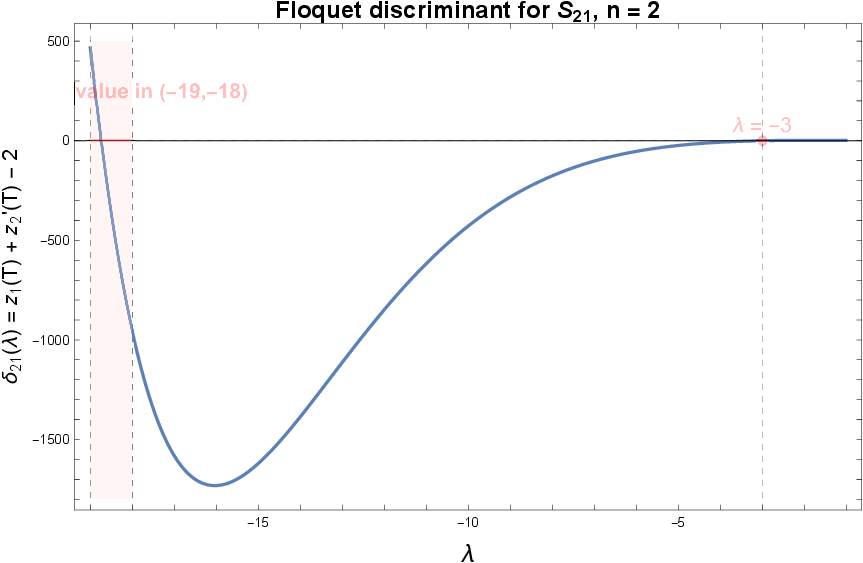}
  \caption{Graph of the Floquet discriminant $\delta_{21}$. This image shows that $-19<\lambda_{21}^1<-18$.}
  \label{neq2S21}
\end{figure}

We will use the fact that if $j>k$ then $\lambda_{ji}^\ell>\lambda_{ki}^\ell$ for any pair of positive integers $i$ and $\ell$. This follows from the Rayleigh characterization of eigenvalues. See \cite{P} for details. Since $-19<\lambda_{21}^1<-18$, then $\lambda_{31}^1>-19$. Figure \ref{neq2S31} shows that $-16<\lambda_{31}^1<-15$. By plotting values of $\lambda$ between $-19$ and $0$, we can numerically verify that there are no additional negative eigenvalues for $S_{31}$.

Using the same arguments, we obtain that $-10.5<\lambda_{41}^1<-9.5$ and $\lambda_{41}^2>0$ (see Figure \ref{neq2S41}); $-2.5<\lambda_{51}^1<-1.5$ and $\lambda_{51}^2>0$ (see Figure \ref{neq2S51}); and finally, we can show that $\lambda_{61}^1>0$ (see Figure \ref{neq2S61}) and that $\lambda_{32}^1>0$ (see Figure \ref{neq2S32}).

Since $\lambda_{32}^1>0$, the Rayleigh characterization implies that $\lambda_{j2}^1>0$ for every $j\ge 3$. For the pairs of indices considered in these computations, we also observe numerically that $\delta_{ij}(\lambda)=\delta_{ji}(\lambda)$.

\begin{figure}[ht]
  \centering
  \includegraphics[width=0.6\textwidth]{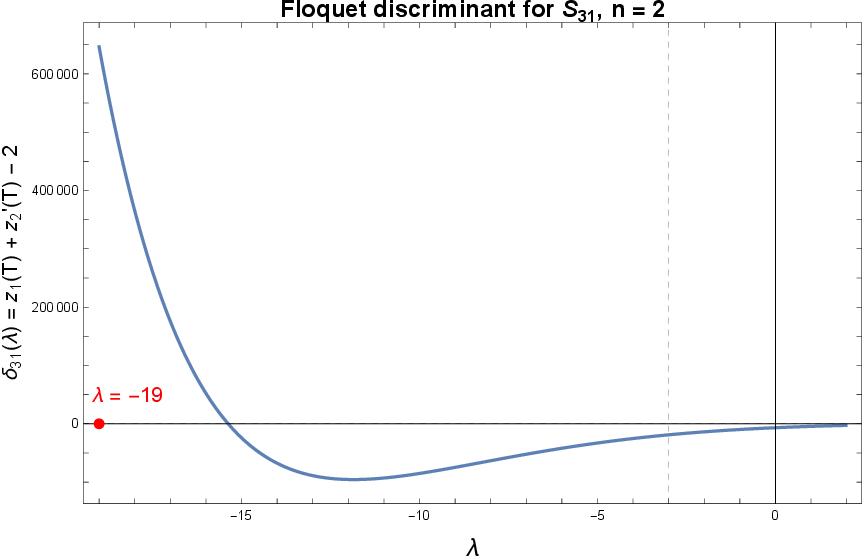}
  \caption{Graph of the Floquet discriminant $\delta_{31}$. This image shows that $-16<\lambda_{31}^1<-15$.}
  \label{neq2S31}
\end{figure} 

\begin{figure}[ht]
  \centering
  \includegraphics[width=0.6\textwidth]{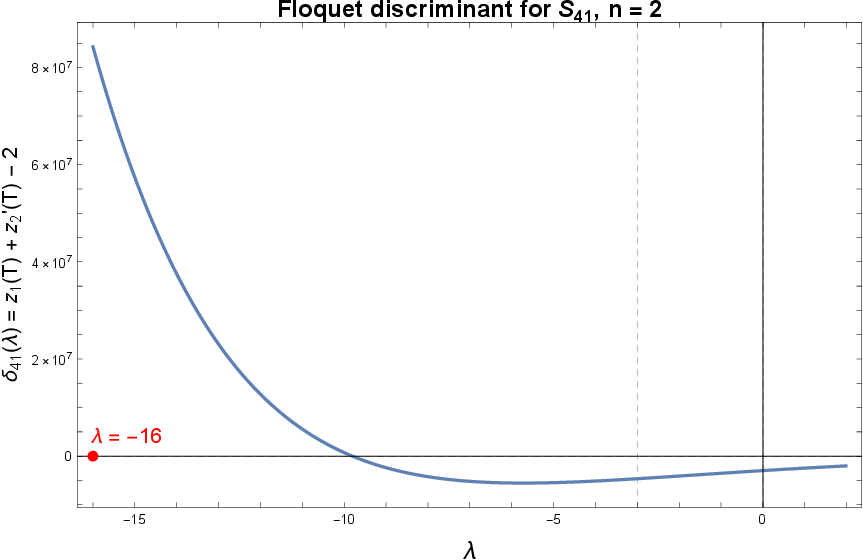}
  \caption{Graph of the Floquet discriminant $\delta_{41}$.  This image shows that $-10.5<\lambda_{41}^1<-9.5$.}
  \label{neq2S41}
\end{figure} 

\begin{figure}[ht]
  \centering
  \includegraphics[width=0.6\textwidth]{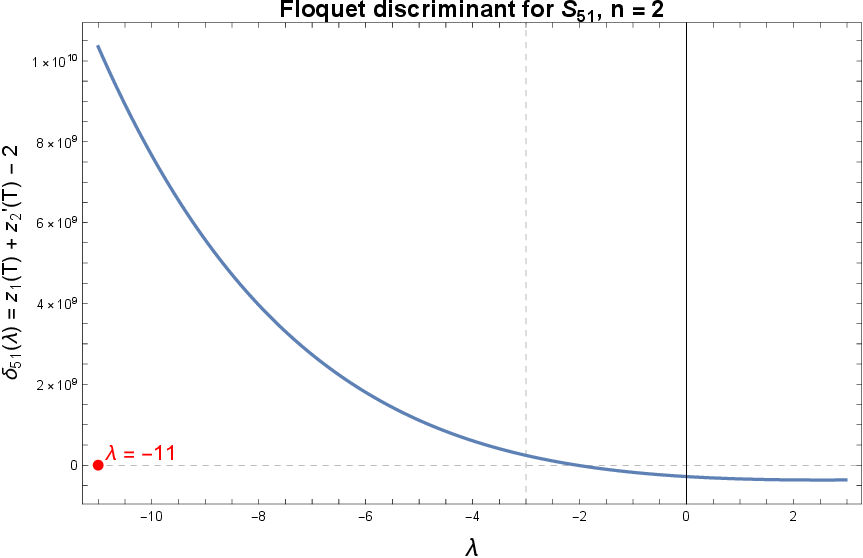}
  \caption{Graph of the Floquet discriminant $\delta_{51}$.  This image shows that $-2.5<\lambda_{51}^1<-1.5$.}
  \label{neq2S51}
\end{figure} 

\begin{figure}[ht]
  \centering
  \includegraphics[width=0.6\textwidth]{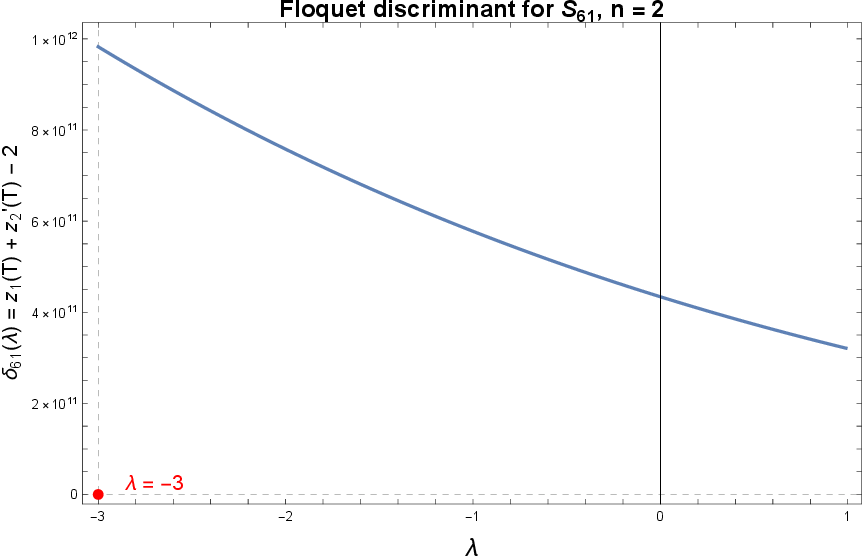}
  \caption{Graph of the Floquet discriminant $\delta_{61}$.  This image shows that $\lambda_{61}^1>0$.}
  \label{neq2S61}
\end{figure} 

\begin{figure}[ht]
  \centering
  \includegraphics[width=0.6\textwidth]{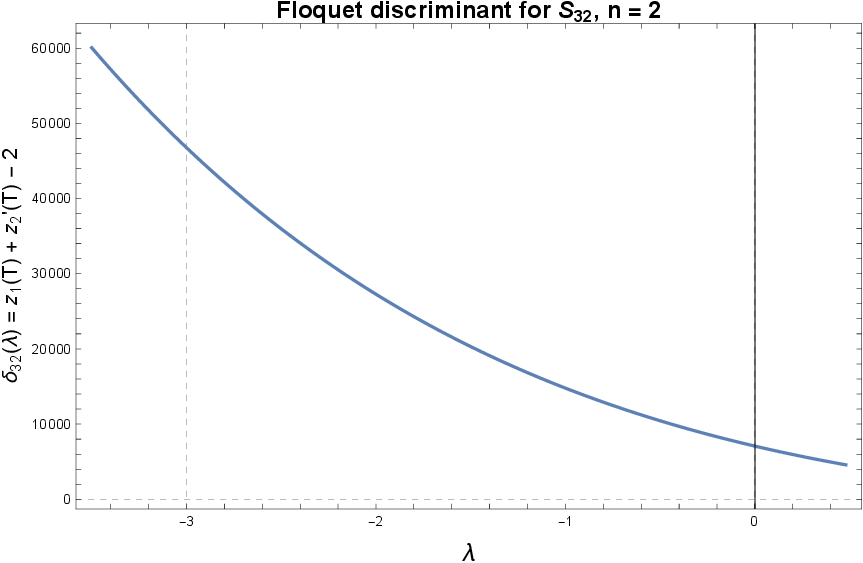}
  \caption{Graph of the Floquet discriminant $\delta_{32}$.  This image shows that $\lambda_{32}^1>0$.}
  \label{neq2S32}
\end{figure} 

Using the Intermediate Value Theorem and the numerical evaluation of the Floquet discriminant, we can locate the negative eigenvalues of the stability operator. More precisely, by detecting sign changes of the discriminant and refining the corresponding intervals, we numerically obtain the following negative eigenvalues together with their multiplicities:

\[
\begin{array}{ccc}
-20.19\ldots \; (\mathrm{mult.}\;1)
&
-19.52\ldots \; (\mathrm{mult.}\;1)
&
-18.74\ldots \; (\mathrm{mult.}\;4)
\\[0.2cm]
-15.37\ldots \; (\mathrm{mult.}\;4)
&
-9.80\ldots \; (\mathrm{mult.}\;4)
&
-3 \; (\mathrm{mult.}\;9,\ \mathrm{expected}\;5)
\\[0.2cm]
-2.05\ldots \; (\mathrm{mult.}\;4)
\end{array}
\]

\subsection{The stability index for $X_{CS}^n$ with $2<n\le 100$} Doing the same analysis, we can check that for values of $n$ with $2<n\le100$ the eigenvalue $\lambda_{51}^1$ becomes positive
and therefore we lose some negative eigenvalues for the stability operator. We did similar verifications for the cases $2<n\le 100$ as we did for the case $n=2$ and we numerically checked that only  the operators $S_{11},S_{22},S_{12}, S_{21},S_{31},S_{13},S_{14},S_{41}$ have negative eigenvalues, each contributing the same amount as in the case $n=2$. Therefore, the stability index of $X_{CS}^n$ is 

{\smaller
$$n^2+4n+3+2 * \frac{n^2+n-2}{2}+2*\frac{n}{6}(n+4)(n-1)=\frac{1}{3} \left(n^3+9 n^2+11 n+3\right)$$
}

Table \ref{table:shooting} contains representative numerical values of
$r_0$ and the quarter period $T/4$ for the embedded closed solutions.
The complete dataset, including the values of $r_0$ and $T/4$ for
$2\le n\le 260$, is available at

\begin{center}
\url{https://www.oscarperdomo.org/CarlottoSchulzExamples}.
\end{center}

\begin{table}[ht]
\centering
\begin{tabular}{ccc}
\hline
$n$ & $r_0$ & $T/4$ \\
\hline
2 & 1.2320724535 & 0.7280351819\\
3 & 1.2804538161 & 0.5765573834\\
4 & 1.3200537364 & 0.4902095753\\
5 & 1.3476517862 & 0.4335387880\\
6 & 1.3679073561 & 0.3927895127\\
7 & 1.3835563354 & 0.3616982744\\
8 & 1.3960950587 & 0.3369751800\\
9 & 1.4064256294 & 0.3167092702\\
10 & 1.4151246829 & 0.2997057081\\
\vdots & \vdots & \vdots\\
100 & 1.5222959845 & 0.0928895608\\
\hline
\end{tabular}
\caption{Representative numerical values of $r_0$ and the quarter period $T/4$ for the embedded closed solutions. The complete dataset is available at \href{https://www.oscarperdomo.org/CarlottoSchulzExamples}{data link}.}
\label{table:shooting}
\end{table}

\section{Verifying the first eigenvalue of the Laplacian}

The solution of the system of differential equations in Theorem~\ref{ThmCS} provides two $T$-periodic functions $r(t)$ and $\theta(t)$, together with a function $\alpha(t)$ satisfying $r(T/4)=\pi/2$ and $\alpha(T/4)=0$.

It was shown in \cite{P} that if $z_1(t)$ denotes the solution of
\[
z_1''(t)+\frac{n-1}{2}\bigl(\ln(EG)\bigr)'(t)z_1'(t)+(2n-1)z_1(t)=0
\]
with initial conditions $z_1(T/4)=1$ and $z_1'(T/4)=0$, then the first nonzero eigenvalue of the Laplace operator is $2n-1$ if and only if $z_1'(5T/4)>0$.

Once numerical approximations of $r_0$ and $T/4$ are available, this criterion can be checked directly. For each value of $n$, we numerically solve the differential equation satisfied by $z_1$ and evaluate $z_1'(5T/4)$. The values of $r_0$ and $T/4$ used in these computations can be found at \href{https://www.oscarperdomo.org/CarlottoSchulzExamples}{data link}. Table~\ref{table:yau} contains representative values of $z_1'(5T/4)$.

\begin{table}[ht]
\centering
\begin{tabular}{cc}
\hline
$n$ & $z_1'(5T/4)$ \\
\hline
2   & 5.81879\\
3   & 6.30013\\
4   & 7.11311\\
5   & 7.89077\\
10  & 11.6129\\
20  & 16.0353\\
50  & 24.9928\\
100 & 35.0016\\
150 & 44.8270\\
200 & 54.5678\\
\hline
\end{tabular}
\caption{Representative values of $z_1'(5T/4)$ for the
Carlotto--Schulz examples.}
\label{table:yau}
\end{table}

 Since all the computed values are positive, it follows from the
characterization above that the first nonzero eigenvalue of the Laplace
operator is $2n-1$ for all Carlotto-Schulz examples with
$2\le n\le 260$. Therefore, Yau's conjecture holds for all such examples.


\end{document}